\documentclass[12pt,reqno]{amsart}
\usepackage{amssymb,amsmath,amsthm,enumerate,verbatim,bbm}
\usepackage[colorlinks=true]{hyperref}

\DeclareMathOperator{\Ran}{Ran}

\DeclareMathOperator{\Ker}{Ker}
\DeclareMathOperator{\Tr}{Tr}
\DeclareMathOperator{\rank}{rank}
\DeclareMathOperator{\supp}{supp}
\DeclareMathOperator*{\slim}{s-lim}
\DeclareMathOperator{\clos}{clos}
\DeclareMathOperator{\Span}{span}
\DeclareMathOperator*{\Res}{Res}

\renewcommand\Im{\hbox{{\rm Im}}\,}
\renewcommand\Re{\hbox{{\rm Re}}\,}
\newcommand{\abs}[1]{\lvert#1\rvert}
\newcommand{\Abs}[1]{\left\lvert#1\right\rvert}
\newcommand{\norm}[1]{\lVert#1\rVert}

\newcommand{\bbT}{{\mathbb T}}
\newcommand{\bbR}{{\mathbb R}}
\newcommand{\bbC}{{\mathbb C}}
\newcommand{\bbN}{{\mathbb N}}
\newcommand{\bbZ}{{\mathbb Z}}

\newcommand{\calH}{{\mathcal H}}
\newcommand{\calM}{\mathcal{M}}

\newcommand{\calU}{\mathcal{U}}

\numberwithin{equation}{section}

\theoremstyle{plain}
\newtheorem{theorem}{\bf Theorem}[section]
\newtheorem*{theorem*}{Theorem 1.1$'$}
\newtheorem{lemma}[theorem]{\bf Lemma}
\newtheorem{proposition}[theorem]{\bf Proposition}

\theoremstyle{definition}
\newtheorem{definition}[theorem]{\bf Definition}

\theoremstyle{remark}
\newtheorem*{remark*}{\bf Remark}
\newtheorem{remark}[theorem]{\bf Remark}

\newcommand{\wt}{\widetilde}
\newcommand{\eps}{\varepsilon}
\newcommand{\elplus}{{\ell^{1,\infty}_{++}}}
\newcommand{\talpha}{{S^*\!{\alpha}}}
\newcommand{\1}{\mathbbm{1}}

\begin{document}

\title[Inverse problem for Hankel operators]{An inverse problem for self-adjoint positive Hankel operators}

\author{Patrick Gerard}
\address{Universit\'e Paris-Sud XI, Laboratoire de Math\'ematiques d'Orsay, CNRS, UMR 8628, et Institut Universitaire de France}
\email{Patrick.Gerard@math.u-psud.fr}

\author{Alexander Pushnitski}
\address{Department of Mathematics, King's College London, Strand, London, WC2R~2LS, U.K.}
\email{alexander.pushnitski@kcl.ac.uk}



\begin{abstract}
For a sequence $\{\alpha_n\}_{n=0}^\infty$, we consider the
Hankel operator $\Gamma_\alpha$, realised as the infinite matrix  in $\ell^2$
with the entries $\alpha_{n+m}$.
We consider the subclass of bounded Hankel operators defined by
the ``double positivity'' condition $\Gamma_\alpha\geq0$,
$\Gamma_{S^*\alpha}\geq0$; here $S^*\alpha$ is the
shifted sequence $\{\alpha_{n+1}\}_{n=0}^\infty$.
We prove that in this class, the sequence $\alpha$ is uniquely determined by
the spectral shift function $\xi_\alpha$ for the pair
$\Gamma_\alpha^2$, $\Gamma_{S^*\alpha}^2$.
We also describe the class of all functions $\xi_\alpha$ arising
in this way and prove that the map $\alpha\mapsto\xi_\alpha$
is a homeomorphism in appropriate topologies.
\end{abstract}

\maketitle

\section{Introduction}\label{sec.a}

\subsection{Hankel operators}

Let $\alpha=\{\alpha_n\}_{n=0}^\infty$
be a bounded sequence of complex numbers.
The Hankel operator $\Gamma_\alpha$ in $\ell^2\equiv \ell^2(\bbZ_+)$ is
formally defined by
$$
(\Gamma_\alpha x)_n=\sum_{m=0}^\infty \alpha_{n+m}x_m,
\quad
x=\{x_n\}_{n=0}^\infty\in\ell^2.
$$
In other words, $\Gamma_\alpha$ is the ``infinite matrix''
\begin{equation}
\Gamma_\alpha
=
\begin{pmatrix}
\alpha_0 & \alpha_1 & \alpha_2 & \ldots
\\
\alpha_1 & \alpha_2 & \alpha_3 & \ldots
\\
\alpha_2 & \alpha_3 & \alpha_4 & \ldots
\\
\vdots & \vdots & \vdots & \ddots
\end{pmatrix}.
\label{a1}
\end{equation}
Background information on Hankel operators can be found, e.g. in \cite{Peller}.
By Nehari's theorem, $\Gamma_\alpha$ is bounded on $\ell^2$ if and only if there exists
a function $f\in L^\infty(\bbT)$ such that $\hat f(n)=\alpha_n$ for all $n\geq0$; here
$\hat f(n)$ is the $n$'th Fourier coefficient of $f$.
We denote by $e_n$, $n\in\bbZ_+$,  the standard basis in $\ell^2$.
Clearly, $\alpha=\Gamma_\alpha e_0$; thus, if $\Gamma_\alpha$ is bounded, then the
sequence $\alpha$ is automatically in $\ell^2$.
In what follows we will always assume that $\Gamma_\alpha$ is bounded and that $\alpha$
is real-valued; in this case $\Gamma_\alpha$ is self-adjoint.

Let $S$ be the (right) shift operator in $\ell^2$:
$$
S:(x_0,x_1,x_2,\ldots)\mapsto (0,x_0,x_1,x_2,\ldots);
$$
we will also need its adjoint,
$$
S^*:(x_0,x_1,x_2,\ldots)\mapsto (x_1,x_2,x_3,\ldots).
$$
Along with $\Gamma_\alpha$, we will consider the Hankel operator $\Gamma_{\talpha}$,
corresponding to the shifted sequence $S^*\alpha$.
In other words, the matrix of $\Gamma_\talpha$ is obtained from the matrix of $\Gamma_\alpha$
by deleting the first row:
\begin{equation}
\Gamma_\talpha
=
\begin{pmatrix}
\alpha_1 & \alpha_2 & \alpha_3 & \ldots
\\
\alpha_2 & \alpha_3 & \alpha_4 & \ldots
\\
\alpha_3 & \alpha_4 & \alpha_5 & \ldots
\\
\vdots & \vdots & \vdots & \ddots
\end{pmatrix}.
\label{a2}
\end{equation}
More formally, $\Gamma_\alpha$ and $\Gamma_\talpha$ are related by
\begin{equation}
\Gamma_\talpha=S^*\Gamma_\alpha=\Gamma_\alpha S.
\label{a3}
\end{equation}
Another important relation between $\Gamma_\alpha$
and $\Gamma_\talpha$ is the formula
\begin{equation}
\Gamma_\talpha^2=\Gamma_\alpha^2-(\cdot,\alpha)\alpha,
\label{a4}
\end{equation}
where $(\cdot,\cdot)$ is the standard inner product in $\ell^2$
(we follow the convention that the inner product is linear in the
first component and anti-linear in the second component).
Formula \eqref{a4} is an elementary consequence of the matrix
representations  \eqref{a1}, \eqref{a2}, or of identities \eqref{a3}.

\subsection{The class $\elplus$}

We will consider positive bounded Hankel operators:
$$
\Gamma_\alpha\geq0.
$$
Here the positivity is understood, as usual, in the quadratic form sense:
$(\Gamma_\alpha x,x)\geq 0$ for all $x\in \ell^2$.
As a consequence of general results of  \cite{MPT},
a positive bounded Hankel operator can have any
continuous spectrum of multiplicity $\leq2$ and any set of non-zero eigenvalues
of multiplicity one (also zero must be in the spectrum
and if zero is an eigenvalue, then it must have infinite multiplicity).
In this paper, we make some progress towards the description of
\emph{isospectral sets}, i.e.\ of the sets of all positive Hankel operators
with a given spectrum.
To simplify the problem, we will consider a special sub-class of positive Hankel operators,
those that satisfy the \emph{double positivity condition:}
\begin{equation}
\Gamma_\alpha\geq 0
\quad \text{ and }\quad
\Gamma_\talpha\geq 0.
\label{a5}
\end{equation}
In Section~\ref{sec.b} we will prove that under this assumption
the non-zero spectrum of $\Gamma_\alpha$ is simple, i.e.\  has multiplicity one.
(In fact, the same applies to $\Gamma_{S^*\alpha}$ and to $\Gamma_{(S^*)^n\alpha}$ for all $n$.)
This property makes the spectral theory of this class of Hankel operators particularly simple.

In order to set up some notation, we first recall that for positive Hankel operators
the operator norm $\norm{\Gamma_\alpha}$ admits a very simple description
in terms of the sequence $\alpha$.
We denote by
$\ell^{1,\infty}$ the set of all sequences $\{x_n\}_{n=0}^\infty$
of complex numbers such that
$$
\norm{x}_{\ell^{1,\infty}}=\sup_{n\geq0} (n+1) \abs{x_n}<\infty.
$$
By \cite[Theorem 3.1]{Widom}, under the positivity assumption $\Gamma_\alpha\geq0$
one has
\begin{equation}
\tfrac14
\norm{\alpha}_{\ell^{1,\infty}}
\leq
\norm{\Gamma_\alpha}
\leq
\pi
\norm{\alpha}_{\ell^{1,\infty}}
\label{a12}
\end{equation}
(the constants are not written explicitly in \cite{Widom} but are easy
to work out from the argument given there; for completeness we give the proof in Section~\ref{sec.b}).
This motivates the following definition.

\begin{definition}
Let $\elplus$ be the class of all sequences $\alpha\in \ell^{1,\infty}$
such that the double positivity conditions \eqref{a5} are fulfilled.
\end{definition}

The double positivity condition implies, in particular,
that the diagonal elements of the ``matrices''
$\Gamma_\alpha$ and $\Gamma_\talpha$ are non-negative;
hence $\alpha_n\geq0$ for all $n\in\bbZ_+$. Notice that, as will be observed in the proof in Section~\ref{sec.b}, the first inequality \eqref{a12} can then be improved as
\begin{equation}
\tfrac12
\norm{\alpha}_{\ell^{1,\infty}}
\leq
\norm{\Gamma_\alpha}\ .
\label{a12bis}
\end{equation}

\subsection{Main results}

In \cite{GG} for $\alpha\in\elplus$ and under the additional assumption
that $\Gamma_\alpha$ and $\Gamma_\talpha$
are compact, it was proven that the spectrum of $\Gamma_\alpha$ \emph{and} the spectrum of
 $\Gamma_\talpha$ \emph{together} uniquely determine the sequence $\alpha$.
 (In fact, the results of \cite{GG} are not limited to the class $\elplus$.)
 In this paper, we consider this inverse spectral problem without the assumption
 of compactness of $\Gamma_\alpha$.
 In this case it turns out that the spectra of $\Gamma_\alpha$ and $\Gamma_\talpha$ in general
 \emph{do not} determine the sequence $\alpha$ (see Section~\ref{sec.f}).
We show that the correct way to parametrise the spectral data is to consider the
\emph{spectral shift function} (SSF)
$$
\xi_\alpha(\lambda)=\xi(\lambda;\Gamma_\alpha^2,\Gamma_\talpha^2).
$$
We refer  to the Appendix for the background information on the SSF theory.
In our case, since (by \eqref{a4}) the difference $\Gamma_\alpha^2-\Gamma_\talpha^2$
is a positive operator of rank one, the SSF satisfies
$$
0\leq \xi_\alpha(\lambda)\leq1, \quad \text{ a.e.\  $\lambda\in\bbR$.}
$$
Further, since both $\Gamma_\alpha^2\ge \Gamma_\talpha^2$ are non-negative
bounded operators,  the SSF is supported on the compact interval $[0,\norm{\Gamma _\alpha}^2]$.

We consider the map
$$
\alpha\mapsto \xi_\alpha;
$$
this gives rise to a \emph{direct problem} (study the properties of $\xi_\alpha$ for a given $\alpha$)
and to an  \emph{inverse problem} (recover $\alpha$ from $\xi_\alpha$).
For convenience, let us introduce a piece of notation for the space where $\xi_\alpha$ is going to live:

\begin{definition}
Let $\Xi_+$ be the set of all functions $\xi\in L^\infty(\bbR)$ with compact support
in $[0,\infty)$ and with values in the interval $[0,1]$.
\end{definition}

Our main result is

\begin{theorem}\label{thm.main}
The map
\begin{equation}
\elplus\ni\alpha\mapsto \xi_\alpha\in\Xi_+
\label{a5a}
\end{equation}
is a bijection between the sets $\elplus$ and $\Xi_+$.
\end{theorem}

Note that this theorem includes two distinct non-trivial statements:
the injectivity and the surjectivity of the map \eqref{a5a}.
Next,
although the inverse map $\xi_\alpha\mapsto\alpha$ is rather complicated,
some information about $\alpha$ can be obtained directly from $\xi_\alpha$.
Indeed, we have two ``trace formulas''
\begin{gather}
\int_0^\infty \xi_\alpha(\lambda)d\lambda = \sum_{n=0}^\infty \alpha_n^2,
\label{a5b}
\\
\frac2\pi \int_0^\infty \left\{1-\exp\left(-\int_0^\infty
\frac{\xi_\alpha(\lambda)}{\lambda+t^2}d\lambda\right)\right\}dt=\alpha_0,
\notag
\end{gather}
see Theorem~\ref{thm.c2}.
We also have an explicit criterion that allows to decide whether
$\Ker\Gamma_\alpha$ is trivial; this happens if and only if both
$$
\int_0^1 \frac{\xi_\alpha(\lambda)}{\lambda}d\lambda=\infty
\text{ and }
\int_0^1 \frac{1-\xi_\alpha(\lambda)}{\lambda}d\lambda=\infty,
$$
see Theorem~\ref{thm.c1}.

Finally, we prove that the map \eqref{a5a} is a homeomorphism with respect to appropriate
weak topologies, which we now introduce.

Given $R>0$, we define the convex subset
$$\elplus (R):=\{ \alpha \in \elplus : \norm{\alpha }_{\ell ^{1,\infty}}\le R\},$$
which we endow with the weak topology relative to the evaluation linear forms $\alpha \mapsto \alpha _n$.
It is well known that $\elplus (R)$ is a metrizable space, and that convergence  of a sequence $\{\alpha^{(p)}\}_{p=1}^\infty$ to $\alpha$ in this space is equivalent to  $\alpha^{(p)}_n\to\alpha_n$ as $p\to\infty$ for all $n$. Notice that, if $R_1<R_2$, $\elplus (R_1)\subset \elplus (R_2)$, and that the corresponding inclusion  is a homeomorphism. Since $\elplus $ is the union of the increasing family $\{ \elplus (R)\} _{R>0}$, we endow it with the inductive limit topology, which is the strongest topology such that, for every $R>0$, the inclusion of $\elplus (R)$ into $\elplus $ is continuous. A sequence is convergent for this topology if, for some $R$, it is contained in $\elplus (R)$, and if it is convergent for the weak topology of $\elplus (R)$. We note that, as it is straightforward to see, the weak convergence $\alpha^{(p)}\to\alpha$ in $\elplus$ implies
norm convergence of $\alpha^{(p)}$ to $\alpha$ in $\ell^r$ for any $r>1$.

We do the same construction with the convex sets $\Xi _+(C)$, $C>0$, corresponding to those elements of $\Xi _+$ which are supported on $[0,C]$. In this case, $\Xi _+(C)$ is endowed with the $L^\infty $ weak* topology, which is known to be metrizable, the convergence of a sequence
$\{\xi^{(p)}\}_{p=1}^\infty$ to $\xi $ being equivalent to
$$
\int_0^\infty \xi^{(p)}(\lambda)\varphi(\lambda)d\lambda
\to
\int_0^\infty \xi(\lambda)\varphi(\lambda)d\lambda
$$
as $p\to\infty$ for all continuous functions $\varphi$. Since  $\Xi _+$ is the union of the increasing family $\{ \Xi _+(C)\}_{C>0}$, we endow it with the inductive limit topology, for which a sequence is convergent if and only if it is contained in some $\Xi _+(C)$ and if it is convergent for the above weak topology of $\Xi _+(C)$.
\begin{theorem}\label{thm.main2}
The map \eqref{a5a} is a homeomorphism with respect to the
above weak topologies.
\end{theorem}

\begin{remark}
There is an interesting analogy between the problem we consider and the
inverse spectral problem for the (singular) Sturm-Liouville operator.
Let $V:\bbR_+\to\bbR$ be a locally integrable function (called the potential); assume that
$V$ is bounded from below.
Consider the differential expression $-\tfrac{d^2}{dx^2}+V(x)$.
Let $H_1$ and $H_2$ be two self-adjoint realisations of this differential
expression in $L^2(\bbR_+)$, corresponding to two distinct choices of
the boundary condition at zero (say, Dirichlet and Neumann). Then the
spectral shift function $\xi(\lambda;H_1,H_2)$ uniquely determines the potential $V$.
If $H_1$, $H_2$ have discrete spectra, this is the classical Borg-Marchenko
result. In general case, this is a result of \cite[Theorem 2.4]{GS}.
\end{remark}

Finally, we would like to mention a number of open questions
related to our results:
\begin{itemize}
\item
How to extend these results to the case of unbounded Hankel operators?
What are the appropriate topologies on the set of sequences $\alpha$ and on
the set of functions $\xi_\alpha$ in this case?
\item
How to extend these results to Hankel operators without the double positivity assumption?
The construction of \cite{GG} suggests that one needs to introduce some additional
spectral variables. An additional problem in this case is that the multiplicity of the
spectrum of $\Gamma_\alpha$ may be non-trivial.
\end{itemize}

\subsection{Some ideas of the proof of Theorem~\ref{thm.main}}
We introduce the spectral measure $\rho_\alpha$ of $\Gamma_\alpha^2$,
corresponding to the element $e_0$:
\begin{equation}
\rho_\alpha(\delta)=(\chi_\delta(\Gamma_\alpha^2) e_0, e_0),
\quad \delta\subset \bbR;
\label{c2a}
\end{equation}
here and in what follows $\chi_\delta$ stands for the characteristic function of
the set $\delta\subset\bbR$. The measure $\rho_\alpha$ is related to $\xi_\alpha$ by
\begin{equation}
z\int_0^\infty \frac{d\rho_\alpha(\lambda)}{\lambda-z}
=
-\exp\left(-\int_0^\infty\frac{\xi_\alpha(\lambda)}{\lambda-z}d\lambda\right)
\label{a11}
\end{equation}
for all $z$ not in the spectrum of $\Gamma_\alpha$; this is an easy calculation
given in Section~\ref{sec.c}.
The proof of injectivity proceeds as follows. Let $\xi_\alpha$ be given;
the relation \eqref{a11} determines
the measure $\rho_\alpha$. Next, we derive a simple recurrence relation
which relates $\rho_{S^*\alpha}$ and $\rho_{\alpha}$:
\begin{equation}
\int_0^\infty \frac{d\rho_\talpha(\lambda)}{\lambda-z}
=
\int_0^\infty \frac{d\rho_\alpha(\lambda)}{\lambda-z}
-
\frac{1}{z}
\left(   \int_0^\infty \frac{d\rho_\alpha(\lambda)}{\lambda-z}  \right)^{-1}
\left(\int_0^\infty \frac{\sqrt{\lambda}d\rho_\alpha(\lambda)}{\lambda-z}\right)^2.
\label{d2a}
\end{equation}
This relation allows one to inductively determine the measures $\rho_{(S^*)^n\alpha}$
for all $n$. This determines the whole sequence $\alpha$ because
\begin{equation}
\alpha_n=(\Gamma_{(S^*)^n\alpha} e_0,e_0)
=
\int_0^\infty \sqrt{\lambda}d\rho_{(S^*)^n\alpha}(\lambda),
\quad n\in \bbN.
\label{d3}
\end{equation}
We also give a second proof of injectivity, which follows the ideas of \cite{GGinv} and \cite{GG},
and is based on the use of the compressed shift operator.

Surjectivity is the hardest statement to prove in Theorem~\ref{thm.main}.
We give two proofs of surjectivity. The first one is based on approximating
a general element $\xi\in\Xi_+$ by elements corresponding to finite rank Hankel operators.
This approach uses the finite rank surjectivity result, which was proven in \cite{GGinv} --- see also \cite{GG}.
Since we use approximation, this approach also relies on Theorem~\ref{thm.main2}.
The second proof of surjectivity is based on the identity \eqref{a11} and on constructing
the measure $\rho_\alpha$, following the path of \cite{MPT}.
Finally, Theorem~\ref{thm.main2} (the continuity of the map \eqref{a5a})
has a surprisingly easy proof based on \eqref{a11},
\eqref{d2a}, \eqref{d3}, on the equivalence of norms \eqref{a12}, and on
trace formula \eqref{a5b}.

\subsection{Some notation}
 Given a sequence $\alpha$, let
\begin{equation}
P_\alpha:\ell^2\to \overline{\Ran\Gamma_\alpha}
\label{a8a}
\end{equation}
be the orthogonal projection.
It is clear from \eqref{a1} and \eqref{a2} that $\Ker \Gamma_\alpha\subset\Ker\Gamma_\talpha$.
Thus, $\overline{\Ran\Gamma_\alpha}$ is an invariant subspace for both operators $\Gamma_\alpha$, $\Gamma_\talpha$,
and on the orthogonal complement to this subspace both operators are equal to  zero. We set
\begin{equation}
\wt\Gamma_\alpha=\Gamma_\alpha|_{\overline{\Ran\Gamma_\alpha}},
\quad
\wt\Gamma_\talpha= \Gamma_\talpha|_{\overline{\Ran\Gamma_\alpha}}.
\label{a9}
\end{equation}
It follows that
$$
\xi_\alpha(\lambda)
=
\xi(\lambda;\wt\Gamma_\alpha^2,\wt\Gamma_\talpha^2).
$$

\section{Operators $\Gamma_\alpha$ with double positivity condition}\label{sec.b}

\subsection{Description in terms of moment sequences}

The second part of the following proposition is borrowed entirely from \cite{Widom}.

\begin{proposition}\label{prp.b1}
Let $\alpha$ be a sequence of real numbers such that the corresponding
Hankel operator $\Gamma_\alpha$ is bounded. Then:
\begin{enumerate}[\rm (i)]
\item
The double positivity condition \eqref{a5} holds true if and only if there exists
a finite positive measure
$\omega$ supported on $[0,1]$ with $\omega(\{1\})=0$
such that $\alpha_n$ can be represented as
\begin{equation}
\alpha_n=\int_0^1 t^n d\omega(t),
\quad
n\geq 0.
\label{a10}
\end{equation}
\item
If $\Gamma_\alpha\geq 0$, then the estimates
\eqref{a12} hold true. If moreover $\Gamma _{\talpha}\geq 0$, then \eqref{a12bis} holds true.
\end{enumerate}
\end{proposition}

\begin{proof}
(i)
Assume that the representation \eqref{a10} holds with some $\omega$.
Then it is evident that
$$
(\Gamma_\alpha x,x)
=
\sum_{n,m=0}^\infty \int_0^1 t^{n+m}x_n\overline{x_m}\, d\omega(t)
=
\int_0^1 \Abs{\sum_{n=0}^\infty t^n x_n}^2 d\omega(t)\geq0,
$$
thus $\Gamma_\alpha\geq0$.
Further, we have
\begin{equation}
\alpha_{n+1}=\int_0^1 t^{n+1}d\omega(t)=\int_0^1 t^n d\omega_1(t),
\quad\text{ where }\quad
d\omega_1(t)=td\omega(t).
\label{b6}
\end{equation}
Thus, by the same reasoning we also have $\Gamma_{S^*\alpha}\geq0$.

Next, assume that the double positivity condition holds true and $\Gamma_\alpha$ is bounded.
By the solution to the classical Hamburger moment problem (see e.g.
\cite[Section X.1, Example 3]{RS2}), condition $\Gamma_\alpha\geq0$
implies that there exists a measure $\omega\geq0$ on $\bbR$ such that
$$
\alpha_n=\int_{-\infty}^\infty t^n d\omega(t),
\quad
n\geq 0.
$$
The boundedness of $\Gamma_\alpha$ implies that $\alpha\in\ell^2$ and then
the measure $\omega$ is unique (see e.g. \cite[Section X.6, Example 4]{RS2}).
By considering even $n$, it is easy to see that the boundedness of $\Gamma_\alpha$
implies that $\supp\omega\subset[-1,1]$ and $\omega(\{-1\})=\omega(\{1\})=0$.
As the same argument applies to $\Gamma_{S^*\alpha}$, we get
that the measure $\omega_1$, given by \eqref{b6}, is also positive.
Thus, $\supp\omega\subset[0,1]$.

(ii) Assume that $\Gamma_\alpha\geq0$.
The second estimate in \eqref{a12} follows directly from
Hilbert's inequality:
$$
\Abs{\sum_{n,m=0}^N\frac{x_n \overline{x_m}}{n+m+1}}
\leq
\pi \sum_{n=0}^N \abs{x_n}^2.
$$
Let us prove the first estimate in \eqref{a12}.
By the proof of (i), we have
$$
\alpha_n=\int_{-1}^1 t^n d\omega(t),
\quad n\geq0
$$
with some finite positive measure $\omega$ such that
$\omega(\{-1\})=\omega(\{1\})=0$.
Fix $\tau\in(0,1)$ and let $x\in\ell^2$ be the element given by $x_n=\tau^n$, $n\geq0$.
We have
\begin{multline*}
(\Gamma_\alpha x,x)
=
\sum_{n,m=0}^\infty \tau^n \tau^m \int_{-1}^1 t^{n+m} d\omega(t)
=
\int_{-1}^1 \frac{d\omega(t)}{(1-t\tau)^2}
\\
\geq
\int_{\tau}^1 \frac{d\omega(t)}{(1-t\tau)^2}
\geq
\frac1{(1-\tau^2)^2}\int_\tau^1 d\omega(t),
\end{multline*}
and therefore
\begin{multline*}
\omega([\tau,1))
\leq
(1-\tau^2)^2(\Gamma_\alpha x,x)
\leq
(1-\tau^2)^2\norm{\Gamma_\alpha}\norm{x}^2
\\
\leq
(1-\tau^2)^2\norm{\Gamma_\alpha}\frac{1}{1-\tau^2}
=
(1-\tau^2)\norm{\Gamma_\alpha}
\leq
2(1-\tau)\norm{\Gamma_\alpha}.
\end{multline*}
Then
$$
\int_0^1 t^n d\omega(t)
=
n\int_0^1 t^{n-1} \omega([t,1))dt
\leq
2\norm{\Gamma_\alpha}n\int_0^1 t^{n-1}(1-t)dt
=
\frac{2}{n+1} \norm{\Gamma_\alpha}.
$$
In the same way, one proves that
$$
\int_{-1}^0 t^n d\omega(t)
\leq
\frac{2}{n+1} \norm{\Gamma_\alpha},
$$
and so
$$
\alpha_n= \int_{-1}^1 t^n d\omega(t)
\leq
\frac{4}{n+1} \norm{\Gamma_\alpha}.
$$
This proves the first estimate in \eqref{a12}. The same proof yields \eqref{a12bis} in the case of the double positivity condition.
\end{proof}

\begin{remark}\label{rmk}
By inspection of the ``matrix'' \eqref{a1} of $\Gamma_\alpha$ we see that
$\Gamma_{(S^*)^2\alpha}$ is a submatrix obtained by deleting the
first row and the first column. It follows that the condition $\Gamma_\alpha\geq0$
implies $\Gamma_{(S^*)^2\alpha}\geq0$ and then, by iteration, $\Gamma_{(S^*)^{2n}\alpha}\geq0$.
Similarly, $\Gamma_\talpha\geq0$ implies $\Gamma_{(S^*)^{2n+1}\alpha}\geq0$
for all $n\in\bbZ_+$.
Thus, the double positivity condition implies that
$\Gamma_{(S^*)^n \alpha}\geq0$ for all $n\in\bbZ_+$.
This can be rephrased as
$$
\alpha\in\elplus
\quad \Rightarrow \quad
(S^*)^n \alpha\in \elplus
\text{ for all $n\in\bbN$.}
$$
This property also follows directly from Proposition~\ref{prp.b1}.
\end{remark}
\begin{remark}\label{rmkbis}
The sequences $(\alpha _n)$ given by \eqref{a10}, for a positive measure $\omega $ on $[0,1]$, are clearly completely monotonic, namely
$$\forall k\ge 0,\  \forall n\ge 0,\  ((I-S^*)^k\alpha )_n\geq 0\ .$$
In the classical paper \cite{H}, Hausdorff proved that this property is in fact equivalent to the representation \eqref{a10}. Hence elements of $\elplus $ are special solutions of the classical Hausdorff moment problem, precisely those which belong to $\ell ^{1,\infty }$.
\end{remark}
\subsection{The simplicity of the spectrum}
We recall that an element $\psi$ of a Hilbert space $\calH$ is called a
\emph{generating element} of a bounded self-adjoint operator $A$ in $\calH$,
if
$$
\calH=\clos\Span\{A^n \psi\mid n=0,1,2,\dots\}.
$$
If $A$ has a generating element, then it has a \emph{simple spectrum},
i.e.\  it is unitarily equivalent to the operator of multiplication by an independent
variable in some $L^2$ space of scalar valued functions.
More precisely: let $\rho_\psi$ be the measure on $\bbR$ defined by
\begin{equation}
\rho_\psi(\delta)=(\chi_\delta(A)\psi,\psi),
\quad \delta\subset \bbR,
\label{b1}
\end{equation}
and let $U_\psi$ be the operator
\begin{equation}
U_\psi:L^2(\bbR,d\rho_\psi)\to \calH, \quad f\mapsto f(A)\psi.
\label{b2}
\end{equation}
Then $U_\psi$ is unitary and
\begin{equation}
U_\psi^*AU_\psi=M_x, \quad U_\psi\1 =\psi,
\label{b3}
\end{equation}
where $M_x$ is the operator of multiplication by the independent variable $x$,
\begin{equation}
(M_x f)(x)=xf(x), \quad x\in \bbR, \quad f\in L^2(\bbR,d\rho_\psi),
\label{b4}
\end{equation}
and $\1$ is the function identically equal to $1$.

\begin{theorem}\label{thm.b1}
Let $\alpha\in\elplus$.
Denote
\begin{align*}
\calM_\alpha&=\clos\Span\{\Gamma_\alpha^n \alpha \mid n=0,1,2,\dots\},
\\
\calM_\talpha&=\clos\Span\{\Gamma_\talpha^n \alpha \mid n=0,1,2,\dots\}.
\end{align*}
Then:
\begin{enumerate}[\rm (i)]
\item
The subspaces $\calM_\alpha$ and $\calM_\talpha$ coincide and will
henceforth be denoted by $\calM$.
\item
$\overline{\Ran \Gamma_\alpha}=\calM$;
the operator $\wt\Gamma_\alpha$ (see \eqref{a9}) has a simple
spectrum and a generating element $P_\alpha e_0$.
\item
$\overline{\Ran \Gamma_\talpha}\subset\calM$;
the operator $\wt\Gamma_\talpha$ has a simple
spectrum and a generating element $\alpha$.
\end{enumerate}
\end{theorem}

\begin{remark*}
\begin{enumerate}
\item
The inclusion
$\overline{\Ran \Gamma_\talpha}\subset\calM$ may be strict.
For example, for $\alpha=e_0$ it is easy to see that
$\Ker \Gamma_\alpha\not=\Ker\Gamma_{S^*\alpha}$, and so
$\overline{\Ran \Gamma_\talpha}\not=\overline{\Ran \Gamma_\alpha}$.
For a description of when this situation occurs, see \cite{GG}.
\item
By Remark~\ref{rmk}, the spectra of all operators
$\Gamma_{(S^*)^n\alpha}$ are simple.
\end{enumerate}
\end{remark*}

In order to prove Theorem~\ref{thm.b1},
first we need a general operator theoretic lemma
which in some form goes back at least
to Kato \cite{Kato} but is probably much older:

\begin{proposition}\label{prp.b2}
Let $A_1$ and $A_0$ be bounded self-adjoint operators in a Hilbert space
such that the difference $A_1-A_0$ is a rank one operator:
\begin{equation}
A_1=A_0+(\cdot,\psi)\psi.
\label{b5}
\end{equation}
Denote
$$
\calM_j=\clos\Span\{A_j^n\psi\mid n=0,1,\dots\}, \quad j=0,1.
$$
Then:
\begin{enumerate}[\rm (i)]
\item
The subspaces $\calM_0$ and $\calM_1$ coincide and will
 henceforth be denoted by $\calM$.
\item
$\calM$ is an invariant subspace both for $A_0$ and for $A_1$.
\item
$A_0|_{\calM^\perp}=A_1|_{\calM^\perp}$.
\end{enumerate}
\end{proposition}
\begin{proof}
(i) Let $f=A_1^n \psi\in\calM_1$. Using \eqref{b5} and expanding,
we see that $f\in\calM_0$. Thus, $\calM_1\subset \calM_0$;
similarly one obtains $\calM_0\subset \calM_1$.
(ii) It is immediate that $A_0(\calM_0)\subset \calM_0$ and
$A_1(\calM_1)\subset \calM_1$.
(iii) If $f\perp\calM$, then in particular $f\perp \psi$. Now apply \eqref{b5}.
\end{proof}

\begin{proof}[Proof of Theorem~\ref{thm.b1}]
Since $\Gamma_\alpha\geq0$, one can approximate (in the operator norm)
odd powers $\Gamma_\alpha^{2n+1}$ by polynomials involving only even powers of $\Gamma_\alpha$.
The same consideration of course applies to $\Gamma_\talpha$. It follows that
$\calM_\alpha$, $\calM_\talpha$ can be rewritten as
\begin{align*}
\calM_\alpha&=\clos\Span\{\Gamma_\alpha^{2n} \alpha \mid n=0,1,2,\dots\},
\\
\calM_\talpha&=\clos\Span\{\Gamma_\talpha^{2n} \alpha \mid n=0,1,2,\dots\}.
\end{align*}
Now let us apply Proposition~\ref{prp.b2} with $A_0=\Gamma_\talpha^2$,
$A_1=\Gamma_\alpha^2$, $\psi=\alpha$.
Part (i) of the Theorem immediately follows from Proposition~\ref{prp.b2}(i).

Next, let $f\perp \calM$; by Proposition~\ref{prp.b2}(iii), we have
$\Gamma_\alpha^2 f=\Gamma_\talpha^2 f$ and therefore,
by the double positivity condition,
we get $\Gamma_\alpha f=\Gamma_\talpha f$.
By \eqref{a3}, this can be rewritten as
$$
\Gamma_\talpha f=S^*\Gamma_\alpha f=\Gamma_\alpha f.
$$
Since $\Ker(S^*-I)=\{0\}$, we obtain $\Gamma_\alpha f=0$ and $\Gamma_\talpha f=0$.
Thus,
$$
\calM^\perp \subset \Ker \Gamma_\alpha
\quad \text{ and } \quad
\calM^\perp \subset \Ker \Gamma_\talpha,
$$
and therefore
$$
\overline{\Ran \Gamma_\alpha}\subset \calM
\quad \text{ and } \quad
\overline{\Ran \Gamma_\talpha}\subset \calM.
$$
Since $\alpha=\Gamma_\alpha P_\alpha e_0$, we also have $\calM\subset \overline{\Ran \Gamma_\alpha}$.
Thus, we get parts (ii) and (iii) of the Theorem.
\end{proof}

\section{Direct spectral problem}\label{sec.c}

\subsection{The perturbation determinant and the trace formulas}
Let $\alpha\in\elplus$.
For $z\notin[0,\infty)$,
consider the perturbation determinant (see Appendix)
for the pair of operators $\Gamma_\alpha^2$,
$\Gamma_\talpha^2$:
$$
\Delta(z)
=
\Delta_{\Gamma_\alpha^2/\Gamma_\talpha^2}(z)
=
\det((\Gamma_\alpha^2-z)(\Gamma_\talpha^2-z)^{-1}).
$$
By \eqref{a4}, it can be explicitly computed as follows:
\begin{multline}
\Delta(z)^{-1}
=
\det((\Gamma_\talpha^2-z)(\Gamma_\alpha^2-z)^{-1})
=
\det(I+(\Gamma_\talpha^2-\Gamma_\alpha^2)(\Gamma_\alpha^2-z)^{-1})
\\
=\det(I-(\cdot,(\Gamma_\alpha^2-\overline{z})^{-1}\alpha)\alpha)
=
1-((\Gamma_\alpha^2-z)^{-1}\alpha,\alpha).
\label{c0}
\end{multline}
Recalling that $\alpha=\Gamma_\alpha e_0$ and using \eqref{A5},
we obtain
\begin{equation}
(\Gamma_\alpha^2(\Gamma_\alpha^2-z)^{-1}e_0,e_0)
=
1-\exp\left(-\int_0^\infty\frac{\xi_\alpha(\lambda)}{\lambda-z}d\lambda\right),
\quad
z\notin[0,\infty).
\label{c1}
\end{equation}
This is one of the key formulas in our construction. It can be alternatively
written as
\begin{equation}
z((\Gamma_\alpha^2-z)^{-1}e_0,e_0)
=
-\exp\left(-\int_0^\infty\frac{\xi_\alpha(\lambda)}{\lambda-z}d\lambda\right),
\quad
z\notin[0,\infty).
\label{c2}
\end{equation}
Let $\rho_\alpha$ be the measure on $\bbR$ defined by \eqref{c2a}, i.e.\
$$
\rho_\alpha(\delta)=(\chi_\delta(\Gamma_\alpha^2) e_0, e_0),
\quad \delta\subset\bbR.
$$
Using the measure $\rho_\alpha$, we may rewrite \eqref{c1} as
\begin{equation}
\int_0^\infty \frac{\lambda d\rho_\alpha(\lambda)}{\lambda-z}
=
1-\exp\left(-\int_0^\infty\frac{\xi_\alpha(\lambda)}{\lambda-z}d\lambda\right).
\label{c2b}
\end{equation}
This gives a one-to-one correspondence between $\xi_\alpha$ and $\rho_\alpha$.
(In fact, $\rho_\alpha$ is in some respects a more convenient functional parameter than $\xi_\alpha$.)

It will be also convenient to use the following modification of the measure $\rho_\alpha$:
\begin{equation}
\wt \rho_\alpha(\delta)=(\chi_\delta(\wt \Gamma_\alpha^2) P_\alpha e_0,  P_\alpha e_0)
\label{c2c}
\end{equation}
(see \eqref{a8a}, \eqref{a9}). Of course, the difference between the measures
$\rho_\alpha$ and $\wt \rho_\alpha$ is only in the weight at zero:
$$
\wt \rho_\alpha(\delta)=\rho_\alpha(\delta)-\rho_\alpha(\delta\cap\{0\}).
$$
By Theorem~\ref{thm.b1}(ii), the operator $\wt\Gamma_\alpha^2$ is unitarily
equivalent to the operator $M_x$ of multiplication by $x$ in $L^2(\bbR,d\wt\rho_\alpha)$;
thus, $\wt\rho_\alpha$ contains all information about the spectrum of $\wt\Gamma_\alpha$.

\begin{theorem}\label{thm.c2}
Let $\alpha\in\elplus$. Then the identities
\begin{gather}
\int_0^\infty \xi_\alpha(\lambda)d\lambda = \sum_{n=0}^\infty \alpha_n^2,
\label{a6}
\\
\frac2\pi \int_0^\infty \left\{1-\exp\left(-\int_0^\infty \frac{\xi_\alpha(\lambda)}{\lambda+t^2}d\lambda\right)\right\}dt=\alpha_0.
\label{a7}
\end{gather}
hold true.
\end{theorem}
\begin{proof}
Formula \eqref{a6} is a direct consequence of \eqref{A3} and of the identity
$$
\Tr((\cdot,\alpha)\alpha)=(\alpha,\alpha)=\sum_{n=0}^\infty \alpha_n^2.
$$
In order to prove \eqref{a7}, we first notice that $\alpha_0=(\Gamma_\alpha e_0,e_0)$.
We also use the integral representation for the square root:
$$
\Gamma_\alpha=\sqrt{\Gamma_\alpha^2}
=
\frac{2}{\pi}\int_0^\infty \Gamma_\alpha^2(\Gamma_\alpha^2+t^2)^{-1}dt.
$$
Putting this together and combining with \eqref{c1}, we obtain \eqref{a7}.
\end{proof}

\subsection{The kernel of $\Gamma_\alpha$}
We recall that (due to the Beurling theorem, see \cite{Beurling})
the kernel of a Hankel operator is either trivial or infinite dimensional.
Further, by \eqref{a4}, the kernel of $\Gamma_\alpha$ is infinite dimensional if and only if
the kernel of $\Gamma_\talpha$ is infinite dimensional. Below we give a concrete
criterion for this to happen.

\begin{theorem}\label{thm.c1}
Let $\alpha\in\elplus$;
then the kernels of  $\Gamma_\alpha$, $\Gamma_\talpha$ are trivial if and only
if both of the following conditions hold true:
\begin{equation}
\int_0^1 \frac{\xi_\alpha(\lambda)}{\lambda}d\lambda=\infty,
\qquad
\int_0^1 \frac{1-\xi_\alpha(\lambda)}{\lambda}d\lambda=\infty.
\label{c3}
\end{equation}
\end{theorem}

Of course, the integral $\int_0^1$ in \eqref{c3} can be replaced by $\int_0^a$ for
any $a>0$.

\begin{proof}
Exactly as in Theorem 4 of \cite{GG},  we have
$$
\Ker \Gamma_\alpha=\{0\}
\quad\Longleftrightarrow\quad
e_0\in \overline{\Ran\Gamma_\alpha}\setminus{\Ran\Gamma_\alpha}.
$$
Let us express the latter condition in terms of the function in the l.h.s. of \eqref{c2}.
Using the spectral theorem for self-adjoint operators, it is easy to see that the strong limit
$$
\slim_{\eps\to0+} \eps(\Gamma_\alpha^2+\eps)^{-1}
$$
exists and is equal to the orthogonal projection onto $\Ker \Gamma_\alpha$.
Thus, using \eqref{c2},
\begin{multline*}
e_0\in \overline{\Ran\Gamma_\alpha}
\quad\Longleftrightarrow\quad
\lim_{\eps\to 0+}\eps((\Gamma_\alpha^2+\eps)^{-1}e_0,e_0)=0
\\
\quad\Longleftrightarrow\quad
\lim_{\eps\to 0+}\int_0^\infty\frac{\xi_\alpha(\lambda)}{\lambda+\eps}d\lambda=\infty
\quad\Longleftrightarrow\quad
\int_0^\infty\frac{\xi_\alpha(\lambda)}{\lambda}d\lambda=\infty,
\end{multline*}
so we obtain the first of the conditions \eqref{c3}.
Next,
$$
e_0\in \Ran\Gamma_\alpha
\quad\Longleftrightarrow\quad
\norm{\Gamma_\alpha^{-1}e_0}^2<\infty
\quad\Longleftrightarrow\quad
\lim_{\eps\to 0+} ((\Gamma_\alpha^2+\eps)^{-1}e_0,e_0)<\infty
$$
and therefore, using \eqref{c2},
$$
e_0\in \Ran\Gamma_\alpha
\quad\Longleftrightarrow\quad
\lim_{\eps\to 0+}\frac1\eps\exp\left(-\int_0^\infty\frac{\xi_\alpha(\lambda)}{\lambda+\eps}d\lambda\right)<\infty.
$$
Finally,
$$
\frac1\eps\exp\left(-\int_0^\infty\frac{\xi_\alpha(\lambda)}{\lambda+\eps}d\lambda\right)
=
\exp\left(\int_0^1\frac{1-\xi_\alpha(\lambda)}{\lambda+\eps}d\lambda\right)\frac{1}{1+\eps}
\exp\left(-\int_1^\infty\frac{\xi_\alpha(\lambda)}{\lambda+\eps}d\lambda\right),
$$
and therefore
$$
e_0\in \Ran\Gamma_\alpha
\quad\Longleftrightarrow\quad
\int_0^1 \frac{1-\xi_\alpha(\lambda)}{\lambda}d\lambda<\infty.
$$
This yields the second condition \eqref{c3}.
\end{proof}

\begin{remark}
In \cite[Theorem 2]{GG} it was proven that in the case of the compact operators
$\Gamma_\alpha$, $\Gamma_{S^*\alpha}$ with the eigenvalues $\{\lambda_j\}_{j=1}^\infty$
and $\{\mu_j\}_{j=1}^\infty$, the kernel of $\Gamma_\alpha$ is trivial if and only if
both of the following conditions hold:
\begin{equation}
\sum_{j=1}^\infty\left(1-\frac{\mu_j^2}{\lambda_j^2}\right)=\infty,
\qquad
\sup_N\frac{1}{\lambda_{N+1}^2} \prod_{j=1}^N \frac{\mu_j^2}{\lambda_j^2}=\infty.
\label{c4}
\end{equation}
In this case we have
$$
\xi_\alpha(\lambda)
=
\begin{cases}
1, & \mu_j^2\leq \lambda\leq \lambda_j^2 \text{ for some $j$,}
\\
0, & \text{ otherwise.}
\end{cases}
$$
Using this formula and some elementary manipulations,
it is not difficult to check that \eqref{c4} is in fact equivalent to \eqref{c3}.
\end{remark}

\section{Inverse spectral problem: uniqueness}\label{sec.d}

\begin{theorem}\label{thm.d1}
The map
$$
\elplus\ni\alpha\longmapsto \xi_\alpha\in\Xi_+
$$
is injective, i.e.\ the sequence $\alpha$ is uniquely determined by the function $\xi_\alpha$.
\end{theorem}
\begin{proof}[First proof]
Let us derive a recurrence relation for $\rho_\talpha$ in terms of $\rho_\alpha$.
From \eqref{a4} by the resolvent identity \eqref{A6}, \eqref{A7} we get
\begin{equation}
(\Gamma_\talpha^2-z)^{-1}-(\Gamma_\alpha^2-z)^{-1}
=
\frac1{D_\alpha(z)}(\cdot,(\Gamma_\alpha^2-\overline{z})^{-1}\alpha)(\Gamma_\alpha^2-z)^{-1}\alpha,
\label{d1}
\end{equation}
where
\begin{multline*}
D_\alpha(z)
=
1-((\Gamma_\alpha^2-z)^{-1}\alpha,\alpha)
=
1-((\Gamma_\alpha^2-z)^{-1}\Gamma_\alpha e_0,\Gamma_\alpha e_0)
\\
=
-z((\Gamma_\alpha^2-z)^{-1} e_0,e_0)
=
-z \int_0^\infty \frac{d\rho_\alpha(\lambda)}{\lambda-z}.
\end{multline*}
Evaluating the quadratic form of both sides of \eqref{d1} on the element $e_0$,
we obtain
\begin{equation}
\int_0^\infty \frac{d\rho_\talpha(\lambda)}{\lambda-z}
-
\int_0^\infty \frac{d\rho_\alpha(\lambda)}{\lambda-z}
=
-
\frac{1}{z}
\left(   \int_0^\infty \frac{d\rho_\alpha(\lambda)}{\lambda-z}  \right)^{-1}
\left(\int_0^\infty \frac{\sqrt{\lambda}d\rho_\alpha(\lambda)}{\lambda-z}\right)^2.
\label{d2}
\end{equation}
Now we can complete the proof.
It is well known that a finite measure on $\bbR$ is uniquely determined
by its Cauchy transform.
Thus, by \eqref{c2b}, the SSF $\xi_\alpha$ uniquely determines the measure $\rho_\alpha$.
Identity \eqref{d2} allows one to determine $\rho_{(S^*)^n \alpha}$ iteratively for all $n\in\bbZ_+$.
Finally, identity
\eqref{d3}
uniquely determines the whole sequence $\alpha$.
\end{proof}
\begin{proof}[Second proof]
Let $\wt\rho_\alpha$ be the measure given by \eqref{c2c}. By \eqref{c2b},
the measure $\wt\rho_\alpha$ is uniquely determined by $\xi_\alpha$; thus, it
suffices to prove that the sequence $\alpha$ is uniquely determined by
the measure $\wt\rho_\alpha$.

For a given sequence $\alpha$, let the unitary operator $U_\alpha$
be given by
\begin{equation}
U_\alpha: L^2(\bbR,d\wt\rho_\alpha)\to \overline{\Ran \Gamma_\alpha},
\quad
f\mapsto f(\wt \Gamma_\alpha^2)P_\alpha e_0.
\label{d6}
\end{equation}
We have (cf. \eqref{b1}--\eqref{b4})
\begin{equation}
U_\alpha^* \wt\Gamma_\alpha^2 U_\alpha=M_x,
\label{d4}
\end{equation}
where $M_x$ is the operator of multiplication by the independent variable in $L^2(\bbR,d\wt\rho_\alpha)$.
Applying the unitary transformation $U_\alpha$ to \eqref{a4},
we get
\begin{equation}
U_\alpha^*\wt \Gamma_\talpha ^2 U_\alpha=M_x-(\cdot, M_x^{1/2}\1)M_x^{1/2}\1.
\label{d5}
\end{equation}
The r.h.s. is an operator in $L^2(\bbR,d\wt\rho_\alpha)$ given by an explicit formula
independent of $\alpha$.
Thus, the operator $U_\alpha^*\wt \Gamma_\talpha ^2 U_\alpha$
(and therefore its square root $U_\alpha^*\wt \Gamma_\talpha U_\alpha$)
is uniquely determined by the measure $\wt\rho_\alpha$.

We will use the the compressed shift
operator
$P_\alpha S P_\alpha^*$.
Denote
$$
\Sigma
=
U_\alpha^*
P_\alpha S P_\alpha^*
U_\alpha.
$$
By \eqref{a3}, we have
$$
\Sigma^* M_{x}^{1/2}=U_\alpha^*\widetilde \Gamma_\talpha U_\alpha,
$$
and therefore the operator $\Sigma^*$ is uniquely determined by $\wt\rho_\alpha$.

By inspection of \eqref{a1}, we find
$$
\alpha_n
=
((S^*)^n \Gamma_\alpha e_0,e_0)
=
((P_\alpha S^* P_\alpha^*)^n \wt\Gamma_\alpha P_\alpha e_0,P_\alpha e_0).
$$
Applying $U_\alpha$, we get
$$
\alpha_n=((\Sigma^*)^nM_{x}^{1/2}\1,\1);
$$
the r.h.s. is uniquely determined by $\wt \rho_\alpha$, and therefore by $\xi_\alpha$.
\end{proof}

\section{Continuity of the map $\alpha\mapsto\xi_\alpha$ and its inverse}\label{sec.g}

Here we prove Theorem~\ref{thm.main2}. It will be useful for us to rephrase it
in a slightly different way; the statement below also includes that the range of $\alpha \mapsto \xi _\alpha $ is closed.
We refer to the introduction for the definition of weak convergence in $\elplus$ and $\Xi_+$. Notice that property \eqref{a12}
implies that $\elplus (R)$ is mapped into $\Xi _+(\pi ^2R^2)$. Conversely,  if $\alpha $ belongs to the inverse image of $\Xi_+(C)$,
using \eqref{c2b}, it is easy to see that the support of  $\rho_{\alpha}$ is contained in $[0,C]$. Since $e_0$ is a generating element of
$\Gamma _\alpha $, we infer that  $\norm{\Gamma _\alpha }\le \sqrt C$, and, by \eqref{a12bis}, that $\alpha \in \elplus(2\sqrt C)$.
In order to prove continuity, it is therefore enough to deal with sequences in $\elplus(R)$ and in $\Xi_+(C)$.

\begin{theorem}\label{thm.g1}
\begin{enumerate}[\rm (i)]
\item
Let $\alpha^{(p)}$ be a sequence of elements in $\elplus (R)$ and let $\alpha$ be a sequence
of real numbers. Assume that  $\alpha^{(p)}_n\to\alpha_n$ as $p\to\infty$ for all $n$.
Then $\alpha\in\elplus$ and $\xi_{\alpha^{(p)}}\to\xi_\alpha$
weakly in $\Xi_+$.
\item
Let $\alpha^{(p)}$ be a sequence of elements of $\elplus$.
If $\xi_{\alpha^{(p)}}\to\xi $ weakly in $\Xi_+$,
then  there exists $\alpha \in \elplus $ such that
$\alpha^{(p)}\to\alpha$ weakly in $\elplus$ as $p\to\infty$, and $\xi =\xi _\alpha $.
\end{enumerate}
\end{theorem}

\begin{proof}
(i) Let $\alpha^{(p)}$, $\alpha$ be as in the hypothesis. As we already mentioned,
 the supports of $\xi_{\alpha^{(p)}}$ all lie in the compact set $[0,\pi ^2R^2]$.

Next, it is straightforward to see that the weak convergence implies that
$\alpha\in\ell^{1,\infty}$ and therefore $\Gamma_\alpha$ is a bounded operator.
Further, since $\ell^{1,\infty}\subset \ell^r$ for any $r>1$,
it is easy to see that $\alpha\in\ell^2$ and we have the norm convergence
$\norm{\alpha^{(p)}-\alpha}_{\ell^2}\to0$ as $p\to\infty$.
We conclude that if
$f$ is a finite linear combination of the elements $\{e_n\}_{n=0}^\infty$, then
$$
\norm{\Gamma_{\alpha^{(p)}}f-\Gamma_\alpha f}_{\ell^2}\to0
\quad
\text{ as $p\to\infty$.}
$$
Using the uniform boundedness of the norms $\norm{\Gamma_{\alpha^{(p)}}}$,
we obtain that
$\Gamma_{\alpha^{(p)}}\to\Gamma_\alpha$ in strong operator topology.
Similarly, $\Gamma_{S^*\alpha^{(p)}}\to\Gamma_{S^*\alpha}$
strongly.
It follows that $\Gamma_\alpha\geq0$ and $\Gamma_{S^*\alpha}\geq0$ and
so $\alpha\in\elplus$.

The strong convergence of operators yields
(see e.g.\  \cite[Theorem VIII.24(a)]{RS1})
the strong convergence of resolvents.
Thus,
$$
(\Gamma_{\alpha^{(p)}}^2(\Gamma_{\alpha^{(p)}}^2-z)^{-1}e_0,e_0)
\to
(\Gamma_{\alpha}^2(\Gamma_{\alpha}^2-z)^{-1}e_0,e_0),
\quad
\text{ as $p\to\infty$,}
$$
for all $\Im z\not=0$.
By \eqref{c1}, we get
$$
\int_0^\infty \frac{\xi_{\alpha^{(p)}}(\lambda)}{\lambda-z}d\lambda
\to
\int_0^\infty \frac{\xi_{\alpha}(\lambda)}{\lambda-z}d\lambda
\quad
\text{ as $p\to\infty$;}
$$
this yields the weak convergence  $\xi_{\alpha^{(p)}}\to\xi_\alpha$.

(ii)
Let $\alpha^{(p)}$ be as in the hypothesis.
By our definition of weak convergence in $\Xi_+$, we have
$\supp\xi_{\alpha^{(p)}}\subset[0,C]$ for all $p$ and some $C>0$. As we already observed,
 $\rho _{\alpha ^{(p)}}$ is supported on $[0,C]$ and $\alpha ^{(p)}$ belongs to $\elplus (2\sqrt C)$.
Consequently, by a diagonal argument, there exists a subsequence $\alpha ^{(p')}$ such that
$\alpha^{(p')}_n\to\alpha_n$ as $p'\to\infty$ for all $n$. Using part i) of the theorem, we infer that
$\alpha \in \elplus $ and that $\xi_{\alpha^{(p')}}\to\xi_\alpha$
weakly in $\Xi_+$, so that $\xi =\xi _{\alpha }$. By the uniqueness Theorem~\ref{thm.d1}, we conclude that $\alpha $ is unique,
thus the whole sequence $\alpha ^{(p)}$ weakly converges to $\alpha $ in $\elplus $.
\end{proof}

\begin{remark}
Using the same arguments as in the above proof, one could similarly  describe the weak continuity of the map $\alpha\mapsto \rho_\alpha$.
\end{remark}

\section{Inverse spectral problem: surjectivity}\label{sec.e}

Here we prove that the map \eqref{a5a} is surjective. We give two proofs.

\subsection{First proof}
The first proof is based on the following result of \cite{GGinv} --- see also \cite{GG} --- about the finite
rank Hankel operators:

\begin{proposition}[\cite{GGinv}, Corollary 3] \label{prp.GG}
Let $N\in\bbN$ and let $\{\lambda_n\}_{n=1}^N$ and $\{\mu_n\}_{n=1}^N$
be two finite sequences of real numbers such that
$$
0< \mu_N<\lambda_N<\mu_{N-1}<\cdots<\mu_1<\lambda_1.
$$
Then there exists $\alpha\in\elplus$ such that the non-zero eigenvalues of
$\Gamma_\alpha$ coincide with $\{\lambda_n\}_{n=1}^N$
and the non-zero eigenvalues of $\Gamma_{S^*\alpha}$
coincide with $\{\mu_n\}_{n=1}^N$.
\end{proposition}

Note that under the hypothesis of the proposition, we have
$$
\xi_\alpha=\sum_{n=1}^N \chi_{\Delta_n},
\quad
\Delta_n=[\mu_n^2,\lambda_n^2].
$$
This follows from \eqref{A8}.

\begin{proof}[First proof of surjectivity]
Let $\xi\in\Xi_+$, $\supp \xi\subset[0,C]$.

\emph{Step 1:} let us show that there exists a sequence
$\xi^{(p)}\in\Xi_+$, $p\in\bbN$, with the following properties:
\begin{enumerate}[\rm (i)]
\item
$\xi^{(p)}\to\xi$ weakly in $\Xi_+$;
\item
for each $p$, the function $\xi^{(p)}$ has the following structure:
$$
\xi^{(p)}=\sum_{n=1}^N\chi_{\Delta_n},
$$
where $\{\Delta_n\}_{n=1}^N$
is a finite collection of disjoint closed intervals in $(0,C]$.
\end{enumerate}

First note that finite linear combinations of characteristic functions of
intervals are dense in $L^1(0,C)$. Approximating $\xi$ by such
functions, we can obtain a weakly convergent sequence  in $\Xi_+$.
Thus, it suffices to show that if $\xi=A\xi_\Delta$, where
$0<A\leq1$ and $\Delta\subset(0,\infty)$ is a compact interval, then
a sequence $\xi^{(p)}$ as above can be constructed.
The latter statement is easy to check directly, and is known in signal processing as the pulse width modulation method.
 Indeed, let $\Delta=[a,b]$;
set $L=b-a$ and
$$
\Delta_n=\left[a+\tfrac{n-1}{N}L, a+\tfrac{n-1+A}{N}L\right],
\quad
1\leq n\leq N,
$$
and let $\xi^{(N)}=\sum_{n=1}^N \chi_{\Delta_n}$.
Then it is easy to see that $\xi^{(N)}\to A\chi_\Delta$ in $\Xi_+$.

\emph{Step 2:}
Let $\xi^{(p)}$ be as constructed at the previous step.
For any fixed $p$, let us
write the intervals $\Delta_n$ as $\Delta_n=[\mu_n^2,\lambda_n^2]$
and let us use Proposition~\ref{prp.GG}.
We obtain an element $\alpha^{(p)}\in\elplus$ such that $\xi_{\alpha^{(p)}}=\xi^{(p)}$.
Thus, we obtain a sequence of elements $\alpha^{(p)}\in\elplus$ such that
$\xi_{\alpha^{(p)}}\to \xi$ weakly in $\Xi_+$.

By Theorem~\ref{thm.g1}(ii),  there exists $\alpha \in \elplus $ such that $\xi =\xi _\alpha $, as required.
\end{proof}

\subsection{Second proof}
The second proof of surjectivity is heavily based on the construction of \cite{MPT}.

\emph{Step 1:}
Let $\xi\in\Xi_+$, $\supp \xi\subset[0,C]$.
Consider the analytic function
\begin{equation}
\Delta(z)=\exp\left(\int_0^\infty \frac{\xi(\lambda)}{\lambda-z}d\lambda\right),
\quad
z\notin[0,\infty).
\label{e0a}
\end{equation}
For $z$ in the upper half-plane set $z=x+iy$, $y>0$; then, using the
assumption $0\leq \xi\leq 1$, we have
$$
0\leq y\int_0^\infty \frac{\xi(\lambda)}{(\lambda-x)^2+y^2}d\lambda\leq \pi,
$$
and therefore
$$
\Im \Delta(x+iy)
=
\exp\left(\int_0^\infty \frac{(\lambda-x)\xi(\lambda)}{(\lambda-x)^2+y^2}d\lambda\right)
\sin\left(y \int_0^\infty \frac{\xi(\lambda)}{(\lambda-x)^2+y^2}d\lambda\right)\geq 0.
$$
Thus, $\Delta(z)$ is a Herglotz function (analytic function with a positive imaginary part
in the upper half-plane). Further, one has
$$
\Delta(z)=1-\frac1z \int_0^\infty \xi(\lambda)d\lambda+O(\abs{z}^{-2}),
\quad \abs{z}\to\infty.
$$
It follows that $1-\Delta(z)^{-1}$ is also a Herglotz function which satisfies
$$
1-\Delta(z)^{-1}
=
-\frac1z \int_0^\infty \xi(\lambda)d\lambda+O(\abs{z}^{-2}),
\quad \abs{z}\to\infty.
$$
By a Herglotz representation theorem (see e.g.\  \cite[Section 14]{Levin}), we obtain
\begin{equation}
1-\Delta(z)^{-1}
=
\int_{-\infty}^\infty \frac{d\nu(\lambda)}{\lambda-z},
\label{e0}
\end{equation}
where $\nu\geq0$ is a finite measure on $\bbR$.
Since $\Delta(z)$ is analytic in $\bbC\setminus [0,C]$, we get that
$\supp\nu\subset[0,C]$.
Let $z=-\eps$:
$$
\int_0^\infty \frac{d\nu(\lambda)}{\lambda+\eps}
=
1-\Delta(-\eps)^{-1}
=
1-
\exp\left(-\int_0^\infty \frac{\xi(\lambda)}{\lambda+\eps}d\lambda\right)<1.
$$
It follows that
$$
\int_0^\infty \frac{d\nu(\lambda)}{\lambda}\leq 1.
$$
In particular, this means that $\nu(\{0\})=0$.
Set $d\wt\rho(\lambda)=\lambda^{-1}d\nu(\lambda)$.
Then $\wt\rho$ is a non-negative finite measure with a compact support in $[0,\infty)$
and $\wt\rho(\bbR)\leq 1$.
Rewriting the integral representation \eqref{e0} for $\Delta$ in terms of $\wt\rho$, we obtain
(cf. \eqref{c2b})
\begin{equation}
1-\Delta(z)^{-1}=\int_0^\infty \frac{\lambda d\wt\rho(\lambda)}{\lambda-z}.
\label{e1}
\end{equation}

\emph{Step 2:}
Consider the Hilbert space $L^2(\bbR,d\wt\rho)$ and the operator $M_x$ of multiplication
by the independent variable in this space.
It what follows, it is important that $\wt\rho(\{0\})=0$ and therefore $\Ker M_x$ is trivial.
Set
$$
H_0=M_x^{1/2},
\quad
H=(H_0^2 -(\cdot, H_0\1)H_0\1)^{1/2}.
$$
Below we prove that there exists a
bounded Hankel operator $\Gamma_\alpha$ in $\ell^2$
with $\Gamma_\alpha\geq0$, $\Gamma_\talpha\geq0$ and an isometry
\begin{equation}
U: L^2(\bbR,d\wt\rho)\to \ell^2
\quad \text{ with } \quad
\Ran U=\overline{\Ran \Gamma_\alpha}
\label{e2}
\end{equation}
such that
\begin{equation}
H_0=U^*\wt\Gamma_\alpha U,
\quad
H=U^*\wt\Gamma_\talpha U
\label{e3}
\end{equation}
(compare with \eqref{d6}, \eqref{d4}, \eqref{d5}).

Assume that such operators $\Gamma_\alpha$ and $U$ have already been found.
By \eqref{e1}, we have
$$
\Delta(z)^{-1}=1-((H_0^2-z)^{-1}H_0\1,H_0\1),
$$
and therefore, reversing the steps in calculation \eqref{c0},
$$
\Delta(z)
=
\det((H_0^2-z)(H^2-z)^{-1})
=
\Delta_{H_0^2/H^2}(z).
$$
Thus, from \eqref{e0a} and \eqref{A5}, we find
$$
\xi(\lambda)=\xi(\lambda; H_0^2, H^2).
$$
By \eqref{e3}, this yields $\xi=\xi_\alpha$, as required.

\emph{Step 3:}
Now we need to construct $\Gamma_\alpha$ and $U$ satisfying \eqref{e2} and \eqref{e3}.
The rest of the proof repeats almost verbatim the arguments of \cite[Section III.3]{MPT}.
From the definition of $H$, we have
$$
\norm{H f}^2=(H^2 f,f)\leq (H_0^2 f,f)=\norm{H_0 f}^2
$$
for any $f$. Thus, there exists a contraction $\Sigma_0$ such that
$\Sigma_0 H_0 = H$.
Let $\Sigma=\Sigma_0^*$; then
$$
H_0\Sigma =H.
$$
From here we get
$$
H_0^2-(\cdot,H_0\1)H_0\1
=
H_0\Sigma\Sigma^* H_0,
$$
and therefore we obtain
$$
\Sigma\Sigma^*=I-(\cdot,\1)\1.
$$
For any $f\in L^2(\bbR,d\wt\rho)$ we have
$$
\norm{f}^2=\abs{(f,\1)}^2+\norm{\Sigma^* f}^2.
$$
Iterating this, we obtain
\begin{equation}
\norm{f}^2
=
\sum_{n=0}^\infty \abs{(f,\Sigma^n\1)}^2
+
\lim_{n\to\infty}\norm{(\Sigma^*)^n f}^2,
\label{e4}
\end{equation}
where the limit necessarily exists and the series necessarily converges.
In order to complete the proof, we need
\begin{lemma}\label{lma.e1}
For any $f\in L^2(\bbR,d\wt\rho)$, we have
$$
\lim_{n\to\infty}\norm{(\Sigma^*)^n f}=0.
$$
\end{lemma}

This lemma will be proven at the end of the section.

\emph{Step 4:}
Assuming Lemma~\ref{lma.e1}, let us complete the proof of the theorem.
Consider the map $U:L^2(\bbR,d\wt\rho)\to\ell^2$ given by
$$
f\mapsto\{(f,\Sigma^n \1)\}_{n=0}^\infty.
$$
By \eqref{e4} and Lemma~\ref{lma.e1}, this map is an isometry.
Set
$$
\alpha_n=((\Sigma^*)^nH_0\1,\1).
$$
If $\{e_n\}_{n=0}^\infty$ is the standard basis in $\ell^2$,
we have (by the definition of $U$)
$$
U^* e_n=\Sigma^n \1.
$$
Thus,
$$
(UH_0U^* e_n, e_m)
=
(H_0\Sigma^n \1,\Sigma^m\1).
$$
Using $H_0\Sigma=\Sigma^* H_0$, we get
$$
(H_0\Sigma^n\1,\Sigma^m\1)
=
((\Sigma^*)^n H_0\1,\Sigma^m\1)
=
\alpha_{n+m}.
$$
Similarly, we get
$$
(UH U^*e_n,e_m)
=
(H \Sigma^n \1, \Sigma^m \1)
=
(H_0\Sigma^{n+1}\1,\Sigma^m \1)
=
\alpha_{n+m+1}
=
(\talpha)_{n+m}.
$$
Thus, we obtain
\begin{equation}
UH_0U^*=\Gamma_\alpha,
\quad
UH U^*=\Gamma_\talpha.
\label{e9}
\end{equation}
Since $H_0\geq0$ and $H\geq0$, the double positivity condition for $\Gamma_\alpha$ holds true.
Since $U$ is an isometry, multiplying \eqref{e9} by $U^*$ on the left and by $U$ on the right
gives \eqref{e3}. Finally,
since $\Ker H_0=\{0\}$, from \eqref{e9} we obtain $\Ker \Gamma_\alpha=\Ker U^*$.
This gives the condition $\Ran U=\overline{\Ran \Gamma_\alpha}$.
\qed

\begin{proof}[Proof of Lemma~\ref{lma.e1}]
This is borrowed almost verbatim from \cite{MPT}.

By definition, we have $H^2\leq H_0^2$.
By the Heinz inequality (see e.g.\  \cite[Section~10.4]{BS})
$H^{1/2}\leq H_0^{1/2}$, and therefore
there exists a (unique) contraction $Q$ in $L^2(\bbR,d\wt\rho)$
with
$$
H^{1/2}=Q H_0^{1/2}.
$$
Let us prove that
\begin{equation}
\Ker (Q^*Q-I)=\{0\}.
\label{e5}
\end{equation}
We have
\begin{gather*}
H^2=(H_0^{1/2} Q^*QH_0^{1/2})^2=H_0^{1/2}Q^*QH_0Q^*QH_0^{1/2},
\\
H^2=H_0^2-(\cdot,H_0\1)H_0\1=H_0^{1/2}(H_0-(\cdot,H_0^{1/2}\1)H_0^{1/2}\1)H_0^{1/2}.
\end{gather*}
It follows that
\begin{equation}
Q^*QH_0Q^*Q =H_0-(\cdot,H_0^{1/2}\1)H_0^{1/2}\1,
\label{e6}
\end{equation}
and so
\begin{equation}
(H_0Q^*Qf,Q^*Qf) =(H_0f,f)-\abs{(f,H_0^{1/2}\1)}^2
\label{e7}
\end{equation}
for any $f\in L^2(\bbR,d\wt\rho)$.
We claim that $\Ker (Q^*Q-I)$ is an invariant subspace of $H_0$.
Indeed, if $Q^*Qf=f$, then by \eqref{e7} we have
$f\perp H_0^{1/2}\1$ and so by \eqref{e6}
$$
Q^*QH_0f=H_0f,
$$
i.e.\  $H_0f\in\Ker(Q^*Q-I)$.
Thus, $\Ker (Q^*Q-I)$ is an invariant subspace of $H_0$
which is orthogonal to $H_0^{1/2}\1$.
Then it is orthogonal to the minimal invariant subspace
of $H_0$ that contains $H_0^{1/2}\1$.
Recalling that $H_0=M_x^{1/2}$ we see  (by an approximation argument based on the
Weierstrass theorem) that such minimal subspace coincides with the whole
space $L^2(\bbR,d\wt\rho)$.
Thus we get \eqref{e5}.

Since $0\leq Q^*Q\leq I$ and $\Ker (Q^*Q-I)=\{0\}$,
by the spectral theorem for self-adjoint operators we get
\begin{equation}
\lim_{n\to\infty}
\norm{(Q^*Q)^nf}=0
\quad \forall f\in L^2(\bbR,d\wt\rho).
\label{e8}
\end{equation}
Next, we have
$$
H=H_0^{1/2}Q^*QH_0^{1/2}=\Sigma^*H_0^{1/2}H_0^{1/2},
$$
and therefore
$$
H_0^{1/2} Q^*Q=\Sigma^*H_0^{1/2}.
$$
Iterating the last identity, we obtain
$$
H_0^{1/2}(Q^*Q)^n=(\Sigma^*)^nH_0^{1/2},
$$
and so \eqref{e8} implies that
$$
\lim_{n\to\infty}\norm{(\Sigma^*)^n f}=0
\qquad
\forall f\in\Ran H_0^{1/2}.
$$
Since $\norm{\Sigma^*}\leq1$, by the density
argument we obtain that the last relation in fact holds true for all $f\in L^2(\bbR,d\wt\rho)$.
\end{proof}

\section{Example}\label{sec.f}

Fix a parameter $\gamma>-1/2$. Let
\begin{equation}
\alpha_n=\frac1{n+1+\gamma},
\quad
n=0,1,2,\dots
\label{f1}
\end{equation}
This sequence corresponds to the choice $d\omega(t)=t^\gamma dt$ in \eqref{a10}.
Consider the Hankel operator $\Gamma_\alpha$.
An explicit diagonalisation of $\Gamma_\alpha$ was given by M.~Rosenblum
in \cite{R1}.
This diagonalisation shows, in particular, that the spectrum of $\Gamma_\alpha$
(for any $\gamma>-1/2$) coincides with the interval $[0,\pi]$, is purely absolutely
continuous and has multiplicity one.
Since $S^*\alpha$ also has the form \eqref{f1} with $\gamma$ incremented by $1$,
this yields a whole class of Hankel operators $\Gamma_\alpha$ with identical
spectra of $\Gamma_\alpha$ and $\Gamma_{S^*\alpha}$.
This shows that, unlike in the case of compact Hankel operators, in general the spectra of
$\Gamma_\alpha$ and $\Gamma_{S^*\alpha}$ together do not determine $\alpha$.

Below we give an explicit formula for the measure $\rho_\alpha$
corresponding to the sequence \eqref{f1}.
In order to do this, let us recall Rosenblum's diagonalisation of $\Gamma_\alpha$.
For $k<\frac12+\Re m$, let $W_{k,m}$ be the Whittaker function  (see e.g.\ \cite[Chapter 6]{BE}):
\begin{equation}
\Gamma(m-k+\tfrac12) W_{k,m}(x)
=
x^{m+\frac12}e^{-x/2}
\int_0^\infty e^{-xt} (t+1)^{m+k-\frac12}t^{m-k-\frac12}dt.
\label{f3}
\end{equation}
For $s>0$ and  $f\in L^2(\bbR_+)$, set
$$
(\calU f)(s)=\lim_{\eps\to+0}\int_\eps^\infty \frac1x W_{-\gamma,i\sqrt{s}}(x)f(x)dx.
$$
Then $\calU$ is a unitary operator from $L^2(\bbR_+,dx)$
onto $L^2(\bbR_+,d\mu(s))$, where
$$
d\mu(s)=\frac1{2\pi^2}\sinh(2\pi\sqrt{s})\abs{\Gamma(\tfrac12+\gamma-i\sqrt{s})}^2ds.
$$
Further, for $n=0,1,2,\dots$,
let $L_n$ be the Laguerre polynomial normalised such that the functions
$\phi_n(x)=e^{-x/2}L_n(x)$ form an orthonormal basis in $L^2(\bbR_+)$.
Define a map $V:\ell^2\to L^2(\bbR_+,d\mu(s))$ by setting
$$
(Va)(s)=\calU(\textstyle{\sum_{n=0}^\infty a_n \phi_n}),
\quad
\forall a=\{a_n\}_{n=0}^\infty\in \ell^2.
$$
Then (see \cite{R1} for the proof) $V$ is unitary and $V$ transforms
$\Gamma_\alpha$ into a multiplication operator:
$$
(V\Gamma_\alpha V^{-1} g)(s)=\frac{\pi}{\cosh(\pi \sqrt{s})} g(s).
$$
Thus, we obtain
\begin{equation}
\rho_\alpha(\delta)=(\chi_\delta(\Gamma_\alpha^2)e_0,e_0)
=
\int_0^\infty \chi_\delta\left(\frac{\pi^2}{(\cosh(\pi \sqrt{s}))^2}\right)
\abs{(Ve_0)(s)}^2 d\mu(s),
\label{f2}
\end{equation}
where
\begin{equation}
(Ve_0)(s)=(U\phi_0)(s)=\lim_{\eps\to+0}
\int_\eps^\infty \frac1x W_{-\gamma,i\sqrt{s}}(x)e^{-x/2}dx.
\label{f4}
\end{equation}
The measure $\rho_\alpha$ is purely absolutely continuous and
is supported on the interval $[0,\pi^2]$.
Formula \eqref{f2} can be rewritten as
$$
d\rho_\alpha(\lambda)
=
\frac{1}{\pi}
\abs{(Ve_0)(s)}^2
\abs{\Gamma(\tfrac12+\gamma-i\sqrt{s})}^2
\frac{\sqrt{s}}{\lambda^2}d\lambda,
\quad 0<\lambda<\pi^2,
$$
where the variables $s$ and $\lambda$ are related by
$$
\lambda=\left(\frac{\pi}{\cosh(\pi \sqrt{s})}\right)^2.
$$
In the case $\gamma=0$, the measure $\rho_\alpha$  and even the
corresponding function $\xi_\alpha$ can be computed explicitly.

\begin{lemma}\label{lma.f1}
Let $\alpha$ be defined by \eqref{f1} with $\gamma=0$.
Then $\supp\rho_\alpha=\supp\xi_\alpha=[0,\pi^2]$.
For $\lambda\in(0,\pi^2)$ we have
\begin{equation}
d\rho_\alpha(\lambda)
=
\frac1{\pi^2\sqrt{\lambda}}\cosh^{-1}\left(\tfrac{\pi}{\sqrt{\lambda}}\right)
=
\frac1{\pi^2\sqrt{\lambda}}\log\left(\tfrac{\pi}{\sqrt{\lambda}}+\sqrt{\tfrac{\pi^2}{\lambda}-1}\right)d\lambda
\label{f7}
\end{equation}
and
\begin{equation}
\xi_\alpha(\lambda)
=
\tfrac1\pi \tan^{-1}\left(\tfrac{2}{\pi}\cosh^{-1}\left(\tfrac{\pi}{\sqrt{\lambda}}\right)\right)
=
\tfrac1\pi \tan^{-1}\left(\tfrac{2}{\pi}\log\left(\tfrac{\pi}{\sqrt{\lambda}}+\sqrt{\tfrac{\pi^2}{\lambda}-1}\right)\right).
\label{f6}
\end{equation}
\end{lemma}
\begin{proof}
Substituting the integral representation \eqref{f3} into \eqref{f4}, we find
\begin{multline*}
(Ve_0)(s)
=
\frac{1}{\Gamma(\tfrac12+i\sqrt{s})}
\int_0^\infty x^{i\sqrt{s}-\frac12}e^{-x}
\int_0^\infty e^{-xt}(t+1)^{i\sqrt{s}-\frac12}t^{-i\sqrt{s}-\frac12}dt\, dx
\\
=
\frac{1}{\Gamma(\tfrac12+i\sqrt{s})}
\int_0^\infty (t+1)^{i\sqrt{s}-\frac12}t^{-i\sqrt{s}-\frac12}
\int_0^\infty x^{i\sqrt{s}-\frac12}e^{-x(t+1)} dx\, dt
\\
=
\int_0^\infty (t+1)^{-1}t^{-i\sqrt{s}-\frac12} dt
=
\frac{\pi}{\cosh(\pi \sqrt{s})}.
\end{multline*}
Further, by the reflection formula for Gamma function, we have
$$
\abs{\Gamma(\tfrac12-i\sqrt{s})}^2 = \frac{\pi}{\cosh(\pi \sqrt{s})}.
$$
Thus, we obtain \eqref{f7}.

Next, for  $\arg z\in(0,2\pi)$, denote
$$
\zeta=\frac{\pi}{\sqrt{-z}},
$$
where the branch of the square root is defined so that $\sqrt{-z}>0$ for $z<0$.
Let us prove the formula
\begin{equation}
\exp\left(-\int_0^\infty \frac{\xi_\alpha(\lambda)}{\lambda-z}d\lambda\right)
=
\frac{\sinh^{-1}(\zeta)}{\zeta}
=
\frac{\log(\zeta+\sqrt{\zeta^2+1})}{\zeta}.
\label{f5}
\end{equation}

By \eqref{c2}, we have
\begin{multline*}
\Delta(z)^{-1}
=
\exp\left(-\int_0^\infty \frac{\xi_\alpha(\lambda)}{\lambda-z}d\lambda\right)
=
-z((\Gamma_\alpha^2-z)^{-1}e_0,e_0)
\\
=
-z\int_0^{\infty}(\lambda-z)^{-1}d\rho_\alpha(\lambda)
=
-\frac{z}{\pi^2}
\int_0^{\pi^2}\frac1{(\lambda-z)\sqrt{\lambda}}
\cosh^{-1}(\pi/\sqrt{\lambda})
d\lambda.
\end{multline*}
By a change of variable $x=\cosh^{-1}(\pi/\sqrt{\lambda})$, this transforms into
$$
\Delta(z)^{-1}
=
\frac{2}{\pi}
\int_0^\infty \frac{x\sinh(x)}{(\cosh(x))^2-\frac{\pi^2}{z}}dx
=
\frac{1}{\pi}
\int_{-\infty}^\infty \frac{x\sinh(x)}{(\cosh(x))^2+\zeta^2}dx.
$$
In order to compute the last integral, we write
$$
\pi\Delta(z)^{-1}
=
 \int_{-\infty}^\infty f(x) dx,
\quad
f(x)=\frac{x\sinh(x)}{(\cosh(x))^2+\zeta^2},
$$
and regard it as an integral of the complex variable.
The function $f(x)$, $x\in\bbC$, has poles at $x=i(\pi/2)\pm\sinh^{-1}(\zeta)+2\pi i n$, $n\in\bbZ$.
Let us move the contour of integration from $\bbR$ to $i\pi+\bbR$.
We notice that
$$
\int_{-\infty}^\infty f(x+\pi i )dx=-\int_{-\infty}^\infty f(x)dx.
$$
Moving the contour, we pick up the residues of $f$ at $x_\pm=i(\pi/2)\pm\sinh^{-1}(\zeta)$, and so
by a direct calculation we obtain
$$
2 \int_{-\infty}^\infty f(x) dx
=
2\pi i (\Res_{x_-}f(x)+\Res_{x_+}f(x))
=
2\pi\frac{\sinh^{-1}(\zeta)}{\zeta},
$$
which yields \eqref{f5}.

Now let us prove formula \eqref{f6} for $\xi_\alpha$.
Since $\supp \rho_\alpha\subset[0,\pi^2]$, we also have
$\supp \xi_\alpha\subset[0,\pi^2]$. Denote
$$
F(z)=\int_0^{\pi^2} \frac{\xi_\alpha(\lambda)}{\lambda-z}dz.
$$
By \eqref{f7}, we have
$$
F(z)=-\log\frac{\sinh^{-1}(\zeta)}{\zeta},
\quad
\zeta=\frac{\pi}{\sqrt{-z}}.
$$
Fix $\lambda_0\in(0,\pi^2)$ and let $z=\lambda_0+i0$.
Then
$$
\zeta=i\zeta_0,
\text{ where }
\zeta_0=\frac{\pi}{\sqrt{\lambda_0}}
\text{ and }
\sqrt{\lambda_0}>0.
$$
We have
$$
\xi_\alpha(\lambda_0)
=
\frac1\pi\Im F(\lambda_0+i0)
=
-\frac1\pi \arg\frac{\sinh^{-1}(i\zeta_0)}{i\zeta_0}.
$$
Now it remains to compute the r.h.s.:
$$
\xi_\alpha(\lambda_0)
=
-\frac1\pi \arg\frac{i(\pi/2)+\cosh^{-1}(\zeta_0)}{i\zeta_0}
=
\frac1\pi \arg((\pi/2)+i\cosh^{-1}(\zeta_0)),
$$
which yields \eqref{f6}.
\end{proof}

\appendix
\section{Background information on the SSF theory}
Here, for the reader's convenience, we collect key formulas of the SSF theory
without proofs or references or much discussion. For the details and history we refer
to the survey \cite{BYa} or the book
\cite{Ya}.

Let $A$, $B$ be bounded self-adjiont operators in a Hilbert space.
Assume that $B-A$ is a trace class operator.
Then there exists a real valued function
$\xi\in L^1(\bbR)$ such that the \emph{Lifshits-Krein} trace formula holds true:
\begin{equation}
\Tr(\varphi(B)-\varphi(A))
=
\int_{-\infty}^\infty \xi(\lambda)\varphi'(\lambda)d\lambda,
\quad
\forall \varphi\in C_0^\infty(\bbR).
\label{A1}
\end{equation}
(It is easy to prove that for any $\varphi\in C_0^\infty(\bbR)$, the difference
$\varphi(B)-\varphi(A)$ is a trace class operator.)
This function is called the \emph{spectral shift function} (SSF) for the pair $A$, $B$;
notation: $\xi(\lambda)=\xi(\lambda;B,A)$.
The SSF has the following properties:
\begin{enumerate}[(i)]
\item
If $\pm(B-A)\geq0$, then $\pm\xi(\lambda;B,A)\geq0$ for a.e.\ $\lambda\in\bbR$.

\item
If $\rank(B-A)\leq n$, then $\abs{\xi(\lambda;B,A)}\leq n$ for a.e.\  $\lambda\in\bbR$.

\item
In particular, if $B-A=(\cdot,x)x$ for some element $x$, then $\xi(\lambda;B,A)\in[0,1]$
for  a.e.\  $\lambda\in\bbR$.

\item
One has the estimate
$$
\int_{-\infty}^\infty \abs{\xi(\lambda;B,A)}d\lambda\leq \norm{B-A}_1,
$$
where $\norm{\cdot}_1$ is the trace norm.

\item
One has the identity
\begin{equation}
\int_{-\infty}^\infty \xi(\lambda;B,A)d\lambda=\Tr(B-A).
\label{A3}
\end{equation}

\end{enumerate}
Finally, the SSF is compactly supported; this is a consequence of the boundedness of $B$ and $A$.

If $A$, $B$ are compact operators, then the SSF can be expressed in terms of the eigenvalue
counting functions of $A$, $B$. Denote
$$
N_+(\lambda;A)=\Tr(\chi_{(\lambda,\infty)}(A)), \quad \lambda>0;
$$
then
\begin{equation}
\xi(\lambda;B,A)=N_+(\lambda;B)-N_+(\lambda;A), \quad \lambda>0,
\label{A8}
\end{equation}
with a similar formula for $\lambda<0$. This is a direct consequence of the trace formula \eqref{A1}.

The SSF for the pair $A$, $B$ is closely related to the perturbation determinant $\Delta_{B/A}$
for this pair. The perturbation determinant is defined by
\begin{equation}
\Delta_{B/A}(z)
=
\det(I+(B-A)(A-z)^{-1}),
\label{A4}
\end{equation}
where $z$ is a complex number outside the spectrum of $A$.
The perturbation determinant is an analytic function of $z$, with poles at isolated
eigenvalues of $A$ and zeros at isolated eigenvalues of $B$ (and may have more complicated
singularities at the essential spectra of $A$, $B$). One has
\begin{equation}
\Delta_{B/A}(z)
=
\exp\left\{\int_{-\infty}^\infty \frac{\xi(\lambda)}{\lambda-z}d\lambda\right\}.
\label{A5}
\end{equation}

Finally, we would like to display a resolvent formula for rank one perturbations.
Let $B=A+\gamma(\cdot,x)x$, where $x$ is an element of the Hilbert space and $\gamma\in\bbR$.
Then for all $z$ with $\Im z\not=0$ we have
\begin{align}
(B-z)^{-1}&=(A-z)^{-1}-\frac{\gamma}{D(z)}(\cdot,(A-\overline{z})^{-1}x)(A-z)^{-1}x,
\label{A6}
\\
D(z)&=1+\gamma((A-z)^{-1}x,x).
\label{A7}
\end{align}

\section*{Acknowledgements}
Much of the work on the paper was done during P.G.'s visit to King's College London
and A.P.'s visit to University of Paris-Sud in 2013. The authors are grateful to
both Universities for financial support.

\end{document}